\documentclass{llncs}

\usepackage{tikz}
\usepackage{enumerate,amsmath,amscd,amssymb}
\usetikzlibrary{arrows}

\newcommand{\lang}[1]{\mathcal{L}(#1)}
\newcommand{\Z}{\mathbb{Z}}
\newcommand{\abs}[1]{\vert #1 \vert}
\newcommand{\IDEMP}{\mathcal{IDEMP}}

\begin{document}

\title{A Characterization of Cellular Automata Generated by Idempotents on the Full Shift\thanks{Research supported by the Academy of Finland Grant 131558}}

\titlerunning{Cellular Automata Generated by Idempotents}

\author{Ville Salo\inst{1}}

\institute{University of Turku, Finland, \\
\email{vosalo@utu.fi}
}

\maketitle

\begin{abstract}
In this article, we discuss the family of cellular automata generated by so-called idempotent cellular automata (CA $G$ such that $G^2 = G$) on the full shift. We prove a characterization of products of idempotent CA, and show examples of CA which are not easy to directly decompose into a product of idempotents, but which are trivially seen to satisfy the conditions of the characterization. Our proof uses ideas similar to those used in the well-known Embedding Theorem and Lower Entropy Factor Theorem in symbolic dynamics. We also consider some natural decidability questions for the class of products of idempotent CA.
\end{abstract}

\section{Introduction}
\label{sec:Intro}

Two famous theorems in symbolic dynamics, namely the Embedding Theorem and the Lower Entropy Factor Theorem \cite{LiMa95}, have a similar flavor. In both, we have two subshifts of finite type $X$ and $Y$, such that $h(Y) > h(X)$. We then use the greater entropy of $Y$ to encode every block of $X$ in a suitable size range into a unique block of $Y$ of the same length, such that the corresponding block of $Y$ is always marked by a unique occurrence of an unbordered word $w$. The fact that we find sufficiently many such blocks in $Y$ is a simple consequence of entropy.

The main problem then becomes handling the periodic parts of a point, since in a long subword of period $p$, the words $w$ would need to be at most $p$ apart. This means that the possibility of encoding does not follow from a simple entropy argument. In fact, in both theorems, the necessary and sufficient conditions include an obvious requirement for periodic points, which doesn't automatically follow.

In this article, we solve a third problem using similar argumentation. Unlike the Embedding Theorem and the Lower Entropy Factor Theorem, which are inherently about subshifts, this is a problem for cellular automata: the problem of characterizing the cellular automata $F$ that arise as products of idempotent cellular automata (CA $G$ such that $G^2 = G$). We only consider the case of the full shift in this paper; the case of a mixing SFT would only add some notational overhead, and we will consider this, and further extensions, in a separate paper. It is easy to see that, apart from the trivial case of the identity CA, such a cellular automaton cannot be surjective. The higher entropy of the domain, and an obvious requirement on how $F$ acts on periodic points, are then used to construct $F$ as a product of idempotent CA.

The problem of characterizing the products of idempotent CA arose from its superficial similarity to the well-known open problem of characterizing the products of involutions (CA $G$ such that $G^2 = 1$) \cite{BoDaDo88}. Both problems are about the submonoid of all CA (with respect to composition) generated by a family of CA that, on their own, have very simple dynamics. In fact, just like involutions are the simplest possible type of reversible CA in the sense of generating the smallest possible nontrivial submonoids, idempotents give the simplest possible nontrivial non-surjective dynamics in the same sense. As we shall see, idempotent CA are much easier to handle than involutions, and the obvious necessary condition turns out to be sufficient.

In the process of proving the characterization, we also construct two CA which may be of interest on their own: In Lemma~\ref{lemma:WiggleRoom}, given a non-surjective CA $F$, we construct a non-surjective idempotent CA $E'$ such that $F(E'(x)) = F(x)$. This can be considered a `CA realization' of the Garden of Eden Theorem. Also, from the Marker Theorem, we directly extract a cellular automaton $M$ marking a `not too dense' subset of the coordinates which is `not too sparse' outside periodic parts of the given point. This way we see that the Marker Lemma in its full generality essentially follows from proving it for the full shift, so a uniform set of markers can be effectively constructed which works for even highly uncomputable subshifts. Lemma~\ref{lemma:IdempotentExtensionLemma} may also be of independent interest.

We will show examples of (types of) cellular automata which are not easily decomposable into a product of idempotents, but which are trivially seen to satisfy the conditions of the characterization. Finally, we discuss decidability questions, showing that it is decidable whether a cellular automaton can be decomposed into a product of idempotents, and that many natural questions that are undecidable for one-dimensional CA stay undecidable restricted to products of idempotent CA.

\section{Definitions and Useful Lemmas}

For points $x \in S^\Z$, we use the term \emph{subword} for all contents of finite, one-way infinite and bi-infinite continuous segments that occur in $x$. A subword $u$ is $p$-periodic if $u_i = u_{i + p}$ whenever both $i$ and $i + p$ are indices of $u$, and periodic if it is $p$-periodic for some $p > 0$.

\begin{definition}
A subset $X \subset S^\Z$ is called a \emph{subshift} if it is topologically closed in the product topology of $S^\Z$ and invariant under the left shift. This amounts to taking exactly the points $x \in S^\Z$ not containing an occurrence of a subword from a possibly infinite set of forbidden patterms. If this set of forbidden patterns can be taken to be finite, $X$ is said to be of finite type (an \emph{SFT}).
\end{definition}

In this paper, a \emph{cellular automaton} (or \emph{CA}) is defined as a continuous function between two subshifts $X$ and $Y$ which commutes with the left shifts of $X$ and $Y$. Such functions $F$ are defined by local maps $F_{\mbox{loc}} : S^{[-r, r]} \to S$ by $F(x)_i = F_{\mbox{loc}}(x_{[i-r, i+r]})$. A \emph{radius} of a CA is any $r$ that can be used to define the local map, \emph{the radius} of a CA refers to its minimal radius, and the \emph{neighborhood} of a CA on the full shift is the (unique, relative to $i$) set of cells on which its image at $i$ actually depends. Note that our definition of a cellular automaton does not require the domain and codomain to be equal. The term sliding block code is also used in symbolic dynamics \cite{LiMa95}. We say $F$ is a cellular automaton on the subshift $Z$ if $X = Y = Z$. We denote the identity CA defined by $G(x) = x$ by $\mbox{id}$. If $X$ is the image of an SFT under a cellular automaton, $X$ is said to be \emph{sofic}.

\begin{definition}
The composition, or \emph{product}, of two CA $F$ and $G$ is denoted in the usual way when the range of $G$ coincides with the domain of $F$: $(F \circ G)(x) = F(G(x))$. Note that $F \circ G$ is a cellular automaton.
\end{definition}

\begin{definition}
By $Q_n$ we denote the set of points of $S^\Z$ with \emph{least} period $n$.
\end{definition}

\begin{definition}
By $\IDEMP(X)$, we denote the set of idempotent CA on $X$, that is, CA $G : X \to X$ such that $G^2 = G$. When the alphabet $S$ is obvious from context, we will also write $\IDEMP = \IDEMP(S^\Z)$. Given a subshift $X$ and a class $\mathcal{CLS}$ of cellular automata on $X$, we write $\mathcal{CLS}^*$ for the class of cellular automata on $X$ that appear as products of CA in $\mathcal{CLS}$.
\end{definition}

\begin{definition}
For $u \in S^n$ ($x \in S^\Z$), we write $\lang{u}$ ($\lang{x}$) for the subwords of $u$ (finite subwords of $x$). For a subshift $X \subset S^\Z$, we write $\lang{X} = \bigcup_{x \in X} \lang{x}$.
\end{definition}

\begin{definition}
A set of words $V = \{v_1, \ldots, v_n\}$ is said to be \emph{mutually unbordered} (or $v_1, \ldots, v_n$ are mutually unbordered) if for all $v_i, v_j \in V$
\[ x_{[c_1, c_1+|v_i|-1]} = v_i, x_{[c_2, c_2+|v_j|-1]} = v_j, c_1 \leq c_2 \implies \]
\[ c_2 - c_1 \geq |v_i| \vee (c_1 = c_2 \wedge v_i = v_j) \]
A~word $v$ is said to be \emph{unbordered} if the set $\{v\}$ is mutually unbordered.
\end{definition}

\begin{definition}
We say that a cellular automaton $F$ is preinjective, if for all $x, y \in S^\Z$ such that $x \neq y$, and $x_j = y_j$ for all $|j| \geq N$ for some $N$, we have $F(x) \neq F(y)$.
\end{definition}

\begin{definition}
We say that the subshift $X \subset S^\Z$ is mixing if for all $u, v \in \lang{X}$, and for all sufficiently large $n$, there exists $w$ with $|w| = n$ such that $uwv \in \lang{X}$. It is easy to see that for a mixing SFT $X$, there is a uniform \emph{mixing distance} $m$ such that for any two words $u, v \in \lang{X}$, and for all $n \geq m$, $uwv \in \lang{X}$ for some $w$ with $|w| = n$.
\end{definition}

\begin{definition}
The $k$th SFT approximation of a subshift $X$ is the SFT obtained by allowing exactly the subwords of length $k$ that occur in $X$.
\end{definition}

We will need three classical results from the literature. First, we state the following version of the Garden of Eden theorem. This is a straightforward combination of Theorem~8.1.16 and Corollary~4.4.9 of \cite{LiMa95}.

\begin{lemma}[Garden of Eden Theorem]
Let $X$ be a mixing SFT. A cellular automaton $F : X \to X$ is preinjective if and only it is surjective.
\end{lemma}

For the full shift, the two directions were first proved in \cite{Mo62} and \cite{My63}. We will need both directions of the Garden of Eden Theorem in the proof of Lemma~\ref{lemma:WiggleRoom}.

The following is a version of Lemma~10.1.8 from \cite{LiMa95} where instead of giving a set of cylinders $F$, we give a cellular automaton that, on $x \in X$, mark the cells $i$ such that $\sigma^i(x) \in \bigcup F$ with a $1$, outputting $0$ on all other cells.

\begin{lemma}[Marker Lemma]\cite{LiMa95}
Let $X$ be a shift space and let $N \geq 1$. Then there exists a cellular automaton $M : X \to \{0, 1\}^\Z$ such that
\begin{itemize}
\item the distance between any two $1$'s in $M(x)$ is at least $N$, and
\item if $M(x)_{(i-N, i+N)} = 0^{2N-1}$, then $x_{[i-N, i+N]}$ is $p$-periodic for some $p < N$.
\end{itemize}
\end{lemma}

Our version of the Marker Lemma is clearly equivalent to that of \cite{LiMa95}, but makes it clearer that the marker CA for $S^\Z$ directly works for \emph{all} subshifts of $S^\Z$, since we avoid the explicit use of cylinders, which by definition depend on the subshift $X$. Note that this in particular implies that a uniform set of words defining the cylinders used as markers works for every subshift $X \subset S^\Z$ whether or not $X$ itself is in a any way accessible, and additional complexity in $X$ may not increase the length of these words.

We need the following subset of a lemma from \cite{Bo83} (see also \cite{Ma95}).

\begin{lemma}[Extension Lemma~2.4] \cite{Bo83}
Let $T$, $T'$ and $U$ be subshifts and let $F : T' \to U$ be a CA, so that the following conditions are satisfied:
\begin{itemize}
\item $U$ is a mixing SFT.
\item $T'$ is a subshift of $T$.
\item the period of any periodic point of $T$ is divisible by the period of some periodic point of $U$.
\end{itemize}
Then $F$ can be extended to a CA $G : T \to U$ so that $G|_{T'} = F$.
\end{lemma}

By an application of the Extension Lemma, we obtain a very useful lemma for idempotent CA, which simplifies our construction in Section~\ref{sec:Characterization}.

\begin{lemma}
\label{lemma:IdempotentExtensionLemma}
Let the CA $F : X \to Y$ be surjective and idempotent for a subshift $X \subset S^\Z$ and a mixing subshift $Y \subset X$ containing a unary point (a point ${}^\infty a^\infty$ for $a \in S$). Then there exists an idempotent CA $G : S^\Z \to S^\Z$ such that $G|_X = F|_X$.
\end{lemma}

\begin{proof}
It is easy to see that for any $k$ the $k$th SFT approximation of a mixing subshift is mixing. Since $F$ is idempotent, we have $F|_Y = \mbox{id}|_Y$. Let $r$ be the radius of $F$ and let $U$ be the $(2r+1)$th (mixing) SFT approximation of $Y$, which contains the unary point of $Y$. Note that we obtain an idempotent cellular automaton $F'$ on $X \cup U$ by directly using the local rule of $F$, since points in $X$ map to $Y$, and $F'$ is the identity map on the whole subshift $U$ (since $r$ is the radius of $F$). By the same argument, we may take $F'$ to be a cellular automaton from $X \cup U$ to $U$.

We apply the Extension Lemma to $T = S^\Z$, $T' = X \cup U$, $U$, and the CA $F' : X \cup U \to U$. This gives us a CA $G : S^\Z \to U$ such that $G|_{X \cup U} = F'$. Since $G(x) \in U$ for all $x \in S^\Z$, and $F'$ is the identity map on $U$, it follows that $G$ is idempotent as a cellular automaton on $S^\Z$. On the other hand, $G|_X = F'|_X = F|_X$, which concludes the proof.
\end{proof}

Of course, we could prove a version of Lemma~\ref{lemma:IdempotentExtensionLemma} for extensions to subshifts other than the full shift, as long as the periodic point condition of the Extension Lemma is satisfied.

\section{Cellular Automata Generated by Idempotents on the Full Shift}
\label{sec:Characterization}

We will prove the following theorem in this article.

\begin{theorem}
\label{theorem:Characterization}
$G \in \IDEMP^*$ if and only if
\begin{equation}
\label{eq:IdIfNoWiggle}
\forall n: (G(Q_n) = Q_n \implies G|_{Q_n} = \mbox{id}|_{Q_n}) \wedge (G(S^\Z) = S^\Z \implies G = \mbox{id}).
\end{equation}
\end{theorem}

It is easy to see that `only if' holds.

\begin{lemma}
\label{lemma:OnlyIf}
Let $X \subset S^\Z$ be a subshift and let $G = G_n \circ \cdots \circ G_1$ for some $G_i \in \IDEMP(X)$. Then $G$ satisfies Equation~\eqref{eq:IdIfNoWiggle} where we have $Q_n' = Q_n \cap X$ in place of $Q_n$.
\end{lemma}

\begin{proof}
Let $G(Q_n') = Q_n'$. Then for all $G_i$, also $G_i(Q_n') = Q_n'$ since $Q_n'$ is finite and points can only map from $Q_n'$ to $Q_j'$ with $j \leq n$. But $G_i \in \IDEMP(X)$ so $G_i$ acts as identity on its image, in particular on $Q_n'$, and thus also $G$ acts as identity on $Q_n'$.

Even more obviously, if $G(S^\Z) = S^\Z$ then $G$ acts as identity everywhere.
\end{proof}

It is not hard to show that binary xor-with-right-neighbor on the full shift satisfies the leftmost implication of \eqref{eq:IdIfNoWiggle}, since no $Q_n$ is mapped onto itself. However, it does not satisfy the rightmost implication, so the lhs does not imply the rhs. It is also easy to find a nonsurjective CA the does not satisfy the lhs.

Since we will prove the converse to Lemma~\ref{lemma:OnlyIf} in the rest of this section in the case $X = S^\Z$, assume $G$ satisfies \eqref{eq:IdIfNoWiggle}. It is clear that the identity map is generated by idempotents, so we may assume $G$ is not surjective. By Lemma~\ref{lemma:IdempotentExtensionLemma}, it is enough to show that the cellular automata $F$ we construct are defined, and idempotent, on $Y \cup F(Y)$ where $Y$ the image of the chain of CA constructed sofar.

We will construct $G = F \circ P \circ A \circ E$ as the product of the $4$ CA
\begin{itemize}
\item E, the Garden of Eden CA;
\item A, the Aperiodic Encoder CA;
\item P, the Period Rewriter CA;
\item F, the Finalizer CA.
\end{itemize}
The CA $P$ will be a product of idempotent cellular automata, while the rest are idempotent themselves.

We will dedicate a short subsection to each of these cellular automata, and the crucial idea behind each CA is extracted into a lemma, except for the highly problem-specific $F$. In the case of Section~\ref{sec:Encoding}, this is just the Marker Lemma.

\subsection{Forbidding a Word from the Input: $E$}

Let us start by rewriting the point so that some subword never appears, without changing the image of $G$.

\begin{lemma}
\label{lemma:WiggleRoom}
Let $G' : S^\Z \to S^\Z$ not be surjective. Then there exists an idempotent non-surjective cellular automaton $E'$ such that $G'(E'(x)) = G'(x)$ .
\end{lemma}

\begin{proof}
Let $Z \subsetneq S^\Z$ be the image of $G'$. The Garden of Eden theorem says there is a positive length word $u$ that we can always rewrite to a different word $u'$ with $|u| = |u'|$ without changing the image of $G'$. Clearly we may assume $|u| > 1$. We take one such $u$ and take the automaton $E'$ that rewrites an occurrence of $u$ at $x_{[i, i+|u|-1]}$ to $u'$ if
\begin{itemize}
\item $u$ occurs exactly once in $x_{[i-2|u|+1,i+3|u|-2]}$
\item rewriting $u$ to $u'$ does not introduce a new $u$ overlapping the original occurrence.
\end{itemize}

Assume on the contrary that $E'^2 \neq E'$ and let $E'(x)_{[i, i+|u|-1]} = u$ for $x \in S^\Z$ such that $E'$ rewrites this $u$ to $u'$. The first condition makes sure that at most one rewriting could have happened such that the new $u$ introduced overlaps $[i, i+|u|-1]$. But this means that the second condition could not have been satisfied. Therefore, none of the cells have been rewritten, and necessarily $x_{[i, i+|u|-1]} = u$.

It is impossible for the cells at most $|u|-1$ away from the occurrence of $u$ to have changed, so the first condition was the reason $u$ was not rewritten in the first place, and there is a nearby occurrence of $u$ at $j$ in $x$ preventing this. But then, the two occurrences of $u$ in $i$ and $j$ prevent each other from being rewritten in the whole orbit of the point $x$. This is a contradiction, since we assumed the occurrence at $i$ is rewritten on the second step. This means $E'$ must be idempotent.

Since $E'({}^\infty a u b^\infty) = E'({}^\infty a u' b^\infty)$ for $a \neq u_1$, $b \neq u_{|u|}$, $E'$ is not preinjective, and the other direction of the Garden of Eden theorem says that its image is not the full shift.
\end{proof}

We take $E = E'$, as given by Lemma~\ref{lemma:WiggleRoom} for $G' = G$, as our first idempotent CA. Let $v \notin \lang{E(S^\Z)} \cup \lang{G(S^\Z)}$ and let $Y = \{x \;|\; v \notin \lang{x}\}$. We choose three mutually unbordered words $w, w_0, w_1$ all containing a single copy of $v$ such that $v$ can only overlap $w$, $w_0$ or $w_1$ at its unique occurrence within it. Further, we may assume $Y' = \{x \in S^\Z \;|\; w \notin \lang{x}\}$ is mixing.

\subsection{Encoding Aperiodic Parts and Memorizing Periodic Parts: $A$}
\label{sec:Encoding}

Next, we construct the CA $A$ that, when started from a point not containing the word $v$, marks the borders of long enough periodic subwords (with small enough period) memorizing the repeated pattern, and encodes the aperiodic parts by occurrences of $v$. For this, we need a suitable definition for `long enough periodic subword' and `small enough period'.

Let $m$ be large enough that
\begin{equation}
\label{eq:EncodingLengths}
\abs{\{wuw \;|\; u \in S^{n - 2|w|} \cap \lang{Y'}\}} > \abs{\{u \in \lang{Y} \;|\; |u| = n\}}
\end{equation}
for all $n \geq m$. This is possible by a standard entropy argument since $Y'$ is a mixing SFT and $Y \subsetneq Y'$. Note that since $w$ is unbordered, $w$ occurs only twice in $wuw$ on the LHS. Let $k$ be such that in a word of length $k$, no two distinct periods $p_i, p_j \leq m$ can occur.

Let $y \in S^\Z$, and let $M$ be given by the Marker Lemma for the full shift and $N = m + 1$, and let $M$ have radius $r$. For now, let $r' > 0$ be arbitrary (to be specified later). We construct a shift-commuting function $A$ as follows, applying the rules top-down:
\begin{itemize}
\item If $v$ occurs in $y_{[i-r',i+r']}$, the cell $i$ is not rewritten.
\item If $M(y)_{[i-1,i+2(|w_0|+m+|w_1|)+k-1]} \in 10^*$, the word $y_{[i,i+2(|w_0|+m+|w_1|)+k-1]}$ has a unique period $p \leq m$ by the Marker Lemma and the choice of $k$, and $A$ sandwiches $t = y_{[i, i+p-1]}$ between $w_0$ and $w_1$ rewriting $y_{[i, i+|w_0|+p+|w_1|-1]}$ by $w_0 t w_1$.
\item If $M(y)_{[i-2(|w_0|+m+|w_1|)-k+1,i+1]} \in 0^*1$, the word $y_{[i-2(|w_0|+m+|w_1|)-k+1,i]}$ has a unique period $p \leq m$ by the Marker Lemma and the choice of $k$, and $A$ sandwiches $t = y_{[i-p+1, i]}$ between $w_1$ and $w_0$ rewriting $y_{[i-|w_1|-p-|w_0|+1, i]}$ by $w_1 t w_0$.
\item If $M(y)_{[i,i+n+1]} = 10^n1$ for $n \leq 2(|w_0|+m+|w_1|)+k-2$, $A$ injects $y_{[i,i+n]}$ into a word $wuw$ where $u$ does not contain $w$.
\end{itemize}
The last property is possible by the fact two $1$'s are at least $m + 1$ apart by the Marker Lemma.

We define the aperiodic subwords, the \emph{AS}, of a point $A(E(x))$ as the maximal subwords of the form $wu_1w wu_2w \cdots wu_nw$ (an AS is, formally, a pair containing a word and the index at which it occurs in $A(E(x))$). We define the period bordering subwords, the \emph{PBS}, as the subwords $w_j t w_{1-j}$ (again, also remembering the location). A PBS of the form $w_0tw_1$ is called a \emph{left border}, and a PBS of the form $w_1tw_0$ is called a \emph{right border}. Finally, we define the long periodic subwords, the \emph{LPS}, as the rest of the maximal subwords not intersecting AS or PBS. 

For a sufficiently large choice of $r'$, the restriction $A : E(S^\Z) \cup A(E(S^\Z)) \to A(E(S^\Z))$ is an idempotent CA: First, note that changing $r'$ will only affect the first condition. Consider a rewriting that happens on the second step at $i$. This $i$ must be in an LPS if $r'$ is chosen large enough, since everywhere else, $w$ and thus $v$ occurs with bounded gaps after the application of $A$ by the Marker Lemma. Also, clearly for cells $i$ deep enough (at least $r + |w_0| + m + |w_1|$) inside a $p$-periodic subword with $p \leq m$, $M$ marks no cells with a $1$. It then clear that a large enough choice of $r'$ implies the idempotency of $A$.

Note that, since the length of a minimal $wuw$-pattern is bounded, a CA can determine which type of subword $i$ belongs to, that is, there exists a cellular automaton $T : S^\Z \to \{\mbox{(AS)}, \mbox{(PBS)}, \mbox{(LPS)}\}^\Z$ such that $T(A(E(x)))_i = \mathcal{T}$ if and only if $i$ is in a subword of type $\mathcal{T}$.

We illustrate the structure of a point $(A \circ E)(x)$ in Fig.~\ref{fig:AfterB}.

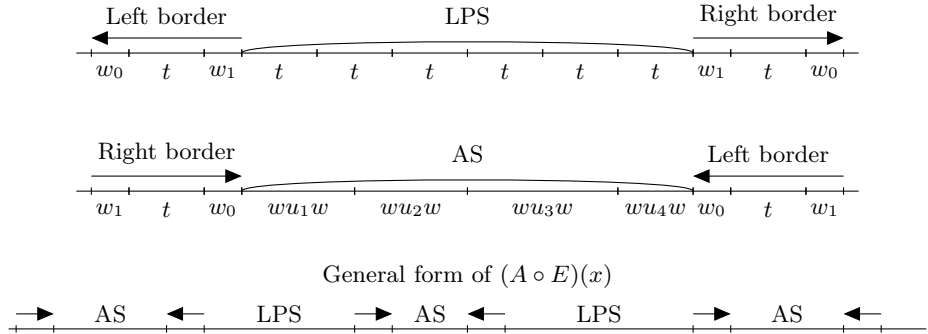
\begin{figure}
\caption{An LPS, an AS, and the general structure of a point, respectively, after $A \circ E$ has been applied. Note that in reality the bordermost copies of $t$ in an LPS are usually cut off (unlike in the figure), and the two $t$ sandwitched between $w_i$ are usually not the same, but rotated versions (conjugates) of each other.}
\begin{center}

\begin{tikzpicture}
  \path[use as bounding box] (-5.2,0.9) rectangle (5.2,-0.9);
  \draw(-5.2,0) -- (5.2,0);
  \foreach \a/\b/\t in {-5/-4.5/$w_0$,-4.5/-3.5/$t$,-3.5/-3/$w_1$,-3/-2/$t$,-2/-1/$t$,-1/0/$t$,0/1/$t$,1/2/$t$,2/3/$t$,3/3.5/$w_1$,3.5/4.5/$t$,4.5/5/$w_0$} {
    \draw(\a,-0.05) -- (\a,0.05);
    \draw(\b,-0.05) -- (\b,0.05);
    \pgfmathparse{0.5*(\a+\b)}
    \draw(\pgfmathresult,-0.25) node {\t};
  }

  \draw(-4,0.5) node{Left border};
  \draw[arrows = {-triangle 45}] (-3,0.2) -- (-5,0.2);
  \draw(0,0.5) node{LPS};
  \draw(-3,0) .. controls +(90:0.2) and +(90:0.2) .. (3,0);
  \draw(4,0.5) node{Right border};
  \draw[arrows = {-triangle 45}] (3,0.2) -- (5,0.2);
\end{tikzpicture}

\begin{tikzpicture}
  \path[use as bounding box] (-5.2,0.9) rectangle (5.2,-0.9);
  \draw(-5.2,0) -- (5.2,0);
  
   \foreach \a/\b/\t in {-5/-4.5/$w_1$,-4.5/-3.5/$t$,-3.5/-3/$w_0$,-3/-1.5/$wu_1w$,-1.5/0/$wu_2w$,0/2/$wu_3w$,2/3/$wu_4w$,3/3.5/$w_0$,3.5/4.5/$t$,4.5/5/$w_1$} {
    \draw(\a,-0.05) -- (\a,0.05);
    \draw(\b,-0.05) -- (\b,0.05);
    \pgfmathparse{0.5*(\a+\b)}
    \draw(\pgfmathresult,-0.25) node {\t};
  }
  
  \draw(-4,0.5) node{Right border};
  \draw[arrows = {-triangle 45}] (-5,0.2) -- (-3,0.2);
  \draw(0,0.5) node{AS};
  \draw(-3,0) .. controls +(90:0.2) and +(90:0.2) .. (3,0);
  \draw(4,0.5) node{Left border};
  \draw[arrows = {-triangle 45}] (5,0.2) -- (3,0.2);
\end{tikzpicture}

\begin{tikzpicture}
  \path[use as bounding box] (-5.2,0.9) rectangle (5.2,-0.5);
  \draw(-6.1,0) -- (6.1,0);
  
  \foreach \x in {-6,-5.5,-4,-3.5,-1.5,-1,0,0.5,3,3.5,5,5.5}
    \draw(\x,-0.05) -- (\x,0.05);
  \foreach \a/\b in {-6/-5.5,-1.5/-1,3/3.5}
    \draw[arrows = {-triangle 45}] (\a,0.2) -- (\b,0.2);
  \foreach \a/\b in {-4/-3.5,0/0.5,5/5.5}
    \draw[arrows = {-triangle 45}] (\b,0.2) -- (\a,0.2);
  \foreach \x/\t in {-4.75/AS,-2.5/LPS,-0.5/AS,1.75/LPS,4.25/AS}
    \draw(\x,0.2) node {\t};
  
  \draw(0,0.7) node{General form of $(A \circ E)(x)$};
  
\end{tikzpicture}

\end{center}
\label{fig:AfterB}
\end{figure}

\subsection{Periodic Subwords of Small Enough Period: $P$}

Now that LPS subwords can be detected by a CA in $(A \circ E)(x)$, we can deal with LPS separately from the rest of the point. So, let us construct a CA $P' \in \IDEMP^*$ which behaves like $G$ on all points with period less than or equal to $m$. We will then modify $P'$ to obtain the desired CA $P$ such that $P \circ A \circ E$ writes the LPS exactly the same way as $G$ would have rewritten the corresponding periodic point (while leaving the original periodic pattern $t$ in the period borders $w_j t w_{1-j}$).

We start with the following lemma, which contains all the essential ideas needed in the construction of $P'$.

\begin{lemma}
\label{lemma:ShiftlessCase}
Let $X$ be a finite set and let $f : X \to X$ not be surjective. Then there exist idempotent functions $f_i$ such that $f = f_n \circ \cdots \circ f_1$.
\end{lemma}

\begin{proof}
First, choose a preimage $g(b) \in f^{-1}(b)$ for all $b \in f(X)$. Then, construct a sequence of idempotent functions $g_i$ that each move a single element $a \in X$ to $g(f(a))$, and leave everything else fixed. Next, move $g(f(X))$ onto $f(X)$ with another product of functions $h_i$. Finally, decompose the permutation of $f(X)$ moving every element $a \in X$ to its final position $f(a)$, into $2$-cycles. Each $2$-cycle can be implemented using three idempotent functions $k_i$ and an element $b \in X - f(X)$. Letting $f = \prod_i k_i \circ \prod_i h_i \circ \prod_i g_i$ completes the construction.
\end{proof}

\begin{lemma}
\label{lemma:PeriodicParts}
Let $m'$ be arbitrary, and let $Y'$ be a subshift such that $G|_{Y'}$ satisfies Equation~\ref{eq:IdIfNoWiggle}. Then there exists
\[ P' \in \IDEMP(Y')^*\]
such that $P'$ acts as $G$ on all points with period less than or equal to $m'$.
\end{lemma}

\begin{proof}
There exists $k'$ such that by looking $k'$ cells in each direction, we can uniquely identify the period of the point. We build $P'$ as the product $P_{m'}' \circ \cdots \circ P_1'$ where each $P_i'$ takes care of points with period $i$. If $G(Q_i \cap Y') = Q_i \cap Y'$ then $P_i'$ is just the identity. All points that map to a point of smaller period simply map directly to that point. This is safe because of the order in which we handle the different periods, since the period of a point cannot be increased by a cellular automaton.

We deal with other points similarly to Lemma~\ref{lemma:ShiftlessCase}, simply shuffling everything in place with a product of idempotents. For this, note that $Q_i \cap Y'$ are partitioned into equivalence classes of size $i$ by the shift, and that an equivalence class either maps to a set of points with smaller period or onto some equivalence class, possibly shifted. This means that the construction in Lemma~\ref{lemma:ShiftlessCase} can be used on equivalence classes: In the terminology of Lemma~\ref{lemma:ShiftlessCase}, the functions $g_i$ are composed with a suitable power of the shift, and finally, additional cellular automata $l_i$ are used to shift the images of all points to their final image (again using a point outside of $f(X)$).
\end{proof}

Let the CA $P' = P_h \circ \cdots \circ P_1$ be given by Lemma~\ref{lemma:PeriodicParts} for $m' = m$ and $Y' = \{x \in S^\Z \;|\; w \notin \lang{x}\}$, where each $P_i$ is idempotent. It is easy to show that if $G$ satisfies Equation~\ref{eq:IdIfNoWiggle} on the full shift, the equation is also satisfied on $Y'$.

To extend $P'$ to $P$, we must make each $P_i$ identify whether the cell being rewritten is part of an LPS. This is complicated by the fact that the intermediate CA $P_i$ may have $v$ in their image. However, since $w$ does not occur in the images of the $P_i$, AS subwords, and thus all types of subwords, are still easy to locate, and the CA $T$ can be extended for this case. If $i$ is not in an LPS, the cell is not rewritten. Otherwise, the cell is rewritten as $P'$ would have, if the periodic pattern were repeated infinitely in both directions. That is, the bordermost $w_1$, if seen, is thought of as a repeater, repeating whichever periodic pattern occurs inside the LPS, see Fig.~\ref{fig:Repeater}. This is possible, since at least $k$ cells are left between the period borders $w_j t w_{1-j}$, and the repeated pattern can be uniquely determined. This concludes the construction of $P$. Note that, as mentioned above, it is enough that the intermediate CA are idempotent on the image $Z$ of the previous chain of CA and their own image from $Z$ by Lemma~\ref{lemma:IdempotentExtensionLemma}.

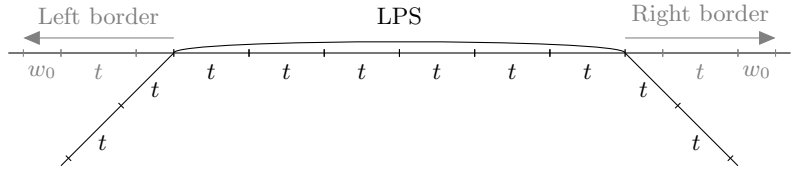
\begin{figure}
\caption{An LPS as seen by $P'$ in $(A \circ E)(x)$.}
\begin{center}

\begin{tikzpicture}
  \path[use as bounding box] (-5.2,0.5) rectangle (5.2,-1.5);
  \draw(-5.2,0) -- (5.2,0);
  \foreach \a/\b/\t in {-5/-4.5/$w_0$,-4.5/-3.5/$t$,-3.5/-3/,3/3.5/,3.5/4.5/$t$,4.5/5/$w_0$} {
    \draw[black!50](\a,-0.05) -- (\a,0.05);
    \draw[black!50](\b,-0.05) -- (\b,0.05);
    \pgfmathparse{0.5*(\a+\b)}
    \draw[black!50](\pgfmathresult,-0.25) node {\t};
  }
  \foreach \a/\b/\t in {-3/-2/$t$,-2/-1/$t$,-1/0/$t$,0/1/$t$,1/2/$t$,2/3/$t$} {
    \draw(\a,-0.05) -- (\a,0.05);
    \draw(\b,-0.05) -- (\b,0.05);
    \pgfmathparse{0.5*(\a+\b)}
    \draw(\pgfmathresult,-0.25) node {\t};
  }
  \draw(-3.0,0) --node[midway,below right=-2]{$t$} (-3.7,-0.7);
  \draw(-3.7,-0.7) +(135:0.05) -- +(135:-0.05);
  \draw(-3.7,-0.7) --node[midway,below right=-2]{$t$} (-4.4,-1.4);
  \draw(-4.4,-1.4) +(135:0.05) -- +(135:-0.05);
  \draw(-4.4,-1.4) -- (-4.5,-1.5);
  
  \draw(3.0,0) --node[midway,below left=-2]{$t$} (3.7,-0.7);
  \draw(3.7,-0.7) +(45:0.05) -- +(45:-0.05);
  \draw(3.7,-0.7) --node[midway,below left=-2]{$t$} (4.4,-1.4);
  \draw(4.4,-1.4) +(45:0.05) -- +(45:-0.05);
  \draw(4.4,-1.4) -- (4.5,-1.5);

  \draw[black!50](-4,0.5) node{Left border};
  \draw[black!50, arrows = {-triangle 45}] (-3,0.2) -- (-5,0.2);
  \draw(0,0.5) node{LPS};
  \draw(-3,0) .. controls +(90:0.2) and +(90:0.2) .. (3,0);
  \draw[black!50](4,0.5) node{Right border};
  \draw[black!50, arrows = {-triangle 45}] (3,0.2) -- (5,0.2);
\end{tikzpicture}
\end{center}
\label{fig:Repeater}
\end{figure}

The only difference in form between $(P \circ A \circ E)(x)$ and $(A \circ E)(x)$ is that the repeating subword of an LPS may have changed, see Fig.~\ref{fig:AfterP}.

\begin{figure}
\caption{An LPS after applying $P$ to $(A \circ E)(x)$.}
\begin{center}

\begin{tikzpicture}
  \path[use as bounding box] (-5.2,0.5) rectangle (5.2,-0.5);
  \draw(-5.2,0) -- (5.2,0);
  \foreach \a/\b/\t in {-5/-4.5/$w_0$,-4.5/-3.5/$t$,-3.5/-3/$w_1$,-3/-2/$t'$,-2/-1/$t'$,-1/0/$t'$,0/1/$t'$,1/2/$t'$,2/3/$t'$,3/3.5/$w_1$,3.5/4.5/$t$,4.5/5/$w_0$} {
    \draw(\a,-0.05) -- (\a,0.05);
    \draw(\b,-0.05) -- (\b,0.05);
    \pgfmathparse{0.5*(\a+\b)}
    \draw(\pgfmathresult,-0.25) node {\t};
  }
 
  \draw(-4,0.5) node{Left border};
  \draw[arrows = {-triangle 45}] (-3,0.2) -- (-5,0.2);
  \draw(0,0.5) node{LPS};
  \draw(-3,0) .. controls +(90:0.2) and +(90:0.2) .. (3,0);
  \draw(4,0.5) node{Right border};
  \draw[arrows = {-triangle 45}] (3,0.2) -- (5,0.2);
\end{tikzpicture}
\end{center}
\label{fig:AfterP}
\end{figure}
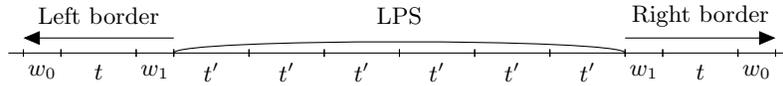

\subsection{The Final Touch: $F$}

Let $l$ be such that if $x_{[i-l,i+l]}$ does not contain $v$ for $x \in S^\Z$, then $(P \circ A \circ E)(x)_i = G(x)_i$. For instance, $l = \max(|w_0|+|w_1|)+r$ has this property, where $r$ is the radius of $P$. All we need to do is rewrite the rest of $x$ as $G$ would have, with a CA $F$. We ensure that $F$ is idempotent by only rewriting $i$ such that $[i - l, i + l]$ contains $v$, since $G(x)$ cannot contain a copy of $v$. But it is easy to deduce the original contents of any cell $j$ that $G$ might use when rewriting such a cell $i$:
\begin{itemize}
\item in an AS, between two $w$, the original contents are given by simply decoding $wuw$.
\item shallow enough inside an LPS, or in a PBS, the $t$ in $w_j t w_{1-j}$ gives the original periodic pattern repeated in $x$.
\end{itemize}

Now,
\[ G = F \circ P \circ A \circ E \]
concludes the proof of Theorem~\ref{theorem:Characterization}.

\section{Examples and Decidability Questions}

\subsection{Examples}

It is now easy to see that while idempotent CA are in some sense trivial, their products can have complicated behavior.

We say that a CA is \emph{nilpotent} if $\exists q, n: \forall x \in S^\Z: G^n(x) = {}^\infty q^\infty$, and we say the CA $F$ has a \emph{spreading state} $q$ if $q$ spreads to $i$ whenever $F$ sees $q$ in the neighborhood of $i$.

\begin{proposition}
\label{proposition:SpreadingState}
If the CA $F$ has a spreading state, has neighborhood size at least $2$, and is constant on unary points, then $F \in \IDEMP^*$.
\end{proposition}

\begin{proof}
Such a CA cannot be preinjective, and thus not surjective either, so the rightmost condition of Theorem~\ref{theorem:Characterization} is satisfied. Also, no $Q_n$ where $n > 1$ maps to itself.
\end{proof}

\begin{proposition}
If a non-surjective CA $F$ has only one spatially and temporally periodic point, then $F \in \IDEMP^*$.
\end{proposition}

By considering north-west deterministic tilings, we find a non-nilpotent cellular automaton on the full shift having a spreading state such that all periodic configurations eventually evolve into the all zero configuration \cite{Ka05}. Such CA are rather nontrivial to construct, and are thus interesting examples of CA in $\IDEMP^*$.

Note that an idempotent CA is simply an eventually periodic automaton with period $1$ and threshold $1$. We say that a CA $G$ is eventually idempotent if the period is $1$, but the threshold need not be, that is, $G^{m + 1} = G^m$ for some $n$. Let us show that such CA are products of idempotent CA.

\begin{proposition}
If $G^{m + 1} = G^m$, then $G \in \IDEMP^*$
\end{proposition}

\begin{proof}
The proof of Lemma~\ref{lemma:OnlyIf} can easily be modified for such CA: If $G(Q_n) = Q_n$ and $G$ is eventually idempotent, we have $G^m(Q_n) = Q_n$ for all $i$. From this, it follows that for all $x \in Q_n$, we have $G(x) = G(G^m(y)) = G^m(y) = x$ for some $y \in Q_n$. The right hand side of Equation~\ref{eq:IdIfNoWiggle} follows similarly.
\end{proof}

\begin{corollary}
If $G$ is a product of eventually periodic CA, then $G \in \IDEMP^*$.
\end{corollary}

\begin{corollary}
Any nilpotent CA $F$ is in $\IDEMP^*$.
\end{corollary}

This means that we have exactly characterized the products of eventually periodic CA with period $1$ and an arbitrary threshold. As we mentioned in Section~\ref{sec:Intro}, the case of period $2$ and threshold $0$ is still open.

\subsection{Decidability Questions}

Although we have complicated examples of CA in $\IDEMP^*$, the problem of whether a CA is in this class is simple to solve using our characterization.

\begin{theorem}
It is decidable whether the CA $F$ is in $\IDEMP^*$.
\end{theorem}

\begin{proof}
Obviously, $F$ being in $\IDEMP^*$ is semi-decidable. On the other hand, if $F$ is not in $\IDEMP^*$, it does not satisfy the characterization of Theorem~\ref{theorem:Characterization}. If $F$ does not satisfy the condition $F(S^\Z) = S^\Z \implies F = \mbox{id}$, the cellular automaton is surjective but not equal to the identity, and since surjectivity and not being equivalent to the identity CA are both semidecidable \cite{AmPa72} \cite{Ka05}, a semialgorithm can detect this. If the condition $\forall n: (F(Q_n) = Q_n \implies F|_{Q_n} = \mbox{id}|_{Q_n})$ is not satisfied, there exists $n$ such that $F(Q_n) = Q_n$, but $F|_{Q_n} \neq \mbox{id}|_{Q_n}$, which is easily found by enumerating the sets $Q_n$.
\end{proof}

However, once restricted to CA in $\IDEMP^*$, we find many undecidable problems, of which we list a few. In \cite{Ka92}, it is shown that nilpotency of cellular automata with a spreading state is undecidable. From this and Proposition~\ref{proposition:SpreadingState}, we obtain the following.

\begin{theorem}
It is undecidable whether $F \in \IDEMP^*$ is nilpotent.
\end{theorem}

By attaching a full shift (with shift dynamics) to the state set so that the spreading state also zeroes cells of the full shift, we obtain that computation of entropy up to error $\epsilon > 0$ is uncomputable even for CA with a spreading state \cite{CuHuKa92}. In particular, we obtain that this is also undecidable for CA in $\IDEMP^*$.

\begin{theorem}
Approximating the entropy of $F \in \IDEMP^*$ up to error $\epsilon$ is uncomputable for all $\epsilon > 0$.
\end{theorem}

\section*{Acknowledgements}

I would like to thank Ilkka T\"orm\"a for his idea of also discussing eventually idempotent cellular automata, and for his useful comments on an early version of this article. I would also like to thank Jarkko Kari for pointing out an error in the examples section.

\bibliographystyle{plain}
\bibliography{bib}{}

\end{document}